\documentclass[final]{siamltex}
\usepackage[leqno]{amsmath}
\usepackage{amssymb,amsmath}
\usepackage{accents}
\usepackage{enumerate}
\usepackage{wasysym}

\title{Blow-up for the one dimensional stochastic wave equations
 }

\begin{document}
\allowdisplaybreaks
\author{WeiJun Deng\footnotemark[2]}

\maketitle
\renewcommand{\thefootnote}{\fnsymbol{footnote}}
\footnotetext[2]{School of Mathematics and Statistics,
 Central South University, Changsha, Hunan 410083, P.R.China
        ({\tt $ weijundeng@csu.edu.cn $}).}

\begin{abstract}
The paper is concerned with the problem of explosive solutions for
a class of semilinear stochastic wave equations.
The challenging open problem(\cite{CMullR}) which is raised by C.Mueller and G.Richards
is included in this  problem.We develop an $\Omega_\delta$-comparative approach.
With the aid of new approach,
under appropriate conditions on the initial data and the nonlinear multiplicative noise
term $(c_2u+f(u)) \dot{W}(t,x)$
with $|f(u)|\geq \kappa |u|^r,r>1,\kappa>0$,
we prove in Theorem 3.1 that the solutions to the stochastic wave equation will blow up in finite
time with positive probability.
\end{abstract}



\newcommand{\R}{\mathbb{R}}
\newcommand{\Z}{\mathbb{Z}}
\newcommand{\N}{\mathbb{N}}
\def\qed{\hfill $\Box$ \smallskip}
\textheight 23cm \textwidth 16cm
\def\baselinestretch{1.3}
\oddsidemargin 0pt \evensidemargin 0pt \topmargin 0pt
\renewcommand{\theequation}{\thesection.\arabic{equation}}
\newcommand{\proofbox}{\hspace*{.2in}\raisebox{3pt}
               {\fbox{\rule{0pt}{.7pt}\rule{.7pt}{0pt}}}}
\baselineskip 8pt
\setlength{\arraycolsep}{4pt}

\newtheorem{thm}{Theorem}[section]
\newtheorem{cor}[thm]{Corollary}
\newtheorem{lem}[thm]{Lemma}
\newtheorem{prop}[thm]{Proposition}
\newtheorem{defn}[thm]{Definition}
\newtheorem{rem}[thm]{Remark}
\newtheorem{ex}[thm]{Example}
\def\lb{\label}
\def\x#1{{\rm (\ref{#1})}}

\section{Introduction}
Consider the initial value problem for the nonlinear stochastic wave equation:
\begin{equation}
\begin{aligned}
&\partial^2_tu=\partial^2_xu+c_1u+(c_2u+f(u)) \dot{W}(t,x),t>0,x\in \overline{D},\\
&u(0,x)=(J+T+1)(1+u_0(x)),\partial_tu(0,x)=(J+T+1)v_0(x).\\
\end{aligned} \label{00}
\end{equation}
Here $f(u)$ is locally Lipschitz function and satisfy
$$|f(u)|\geq \kappa |u|^r,r>1,\kappa>0,$$
$c_1,c_2 $ are given constants,$D=(0,J)\subset R $,
$\dot{W}(t,x)$ is 2-parameter white noise.

In this article we want to study blow-up phenomena:does solutions to stochastic
wave equation(\ref{00}) finite time blow-up occur with positive probability?
It is expected that such a white noise has a strong influence on the solutions which blow-up.
The challenging open problem(\cite{CMullR}) which is raised by C.Mueller and G.Richards is included in this  problem.

For deterministic nonlinear partial differential equations,there is a very extensive literature on
blow-up in finite time.Let us just mention a few:(\cite{JMBall,cakarelliFriedman,HFujita,YGigaRVKohn,RTGlassey,RTGy,TKato,HALevine,HALevi,FMerle}),for example.

On the other hand, for stochastic partial differential equations (SPDE),
there are very few papers about finite time blow-up.
It is mathematically very difficult to consider space-time white noises,this
is due to the lack of smoothing effect in the stochastic differential equation.
We refer the reader to (\cite{CMull,CMul,CM,MDozziJALopez,PLChow,CMu}) for new developments.

Our strategy to study the blow-up is based on the $\Omega_\delta$-comparative approach.
We divide our proof in five steps.

First,we introduce a blow-up lemma for one dimensional semilinear wave equations.
Next,we establish a comparison lemma on semilinear wave equations.
Another step in the proof is we need to verify the essential supremum of the solution of (\ref{00})
over a subset $\Omega-\Omega_\delta$ of the probability space,will blow up in finite time.
We utilize the close relationship between stochastic
partial differential equations and deterministic partial differential equations.
Using reduction to absurdity method,suppose that on the contrary,
the essential supremum of the solution of (\ref{00}) exists for a long time
over a subset $\Omega-\Omega_\delta$ of the probability space.
Consider the deterministic partial differential equations:
\begin{equation}
\begin{array}{l}\partial^2_t u=\partial^2_x u+\frac{\kappa^2}{4}|u|^r-\frac{c_1^2+c_2^2}{2}u,t>0,x\in D,\\
u(t,x)=0,\mbox{for }~x\in \partial D,~t\geq 0,\\
u(0,x)=u_0(x),\partial_t u(0,x)=v_0(x),\label{000}
\end{array}
\end{equation}
suppose that $|u|^r$ increases fast,in other words,$r>1 $.
By the first step,under appropriate conditions on the initial data
solutions of (\ref{000}) will blow up in finite time.
Then we construct a comparison
between square moment of the solution of (\ref{00})
over the subset $\Omega-\Omega_\delta$
and solution of (\ref{000}),
apply the previous comparison lemma to get conflicting results.
The fourth steps to do in the proof is that
the essential infimum of the blow-up time of the solution of (\ref{00}) is bounded.
Finally,we show that the solutions of (\ref{00}) will blow up in bound time with positive probability.

With the aid of the $\Omega_\delta$-comparative approach,
under appropriate conditions on the initial data and the nonlinear multiplicative noise
term $(c_2u+f(u)) \dot{W}(t,x)$
with $|f(u)|\geq \kappa |u|^r,r>1,\kappa>0$,
we prove in Theorem 3.1 that the solution  of the stochastic wave equation will blow up in bound
time with positive probability.

The rest of this paper is organized as follows.
We shall first give problem statement
and preliminaries in Section 2.
Then, in Section 3,we develop a comparative approach and
prove the main theorem (Theorem 3.1).

\section{Problem statement and preliminaries}
\subsection{Problem statement}
Let $\Omega$ be an uncountable Polish space with the metric $\gamma$ and $\mathcal F=\mathcal B(\Omega)$ be the topological $\sigma$-field.
Suppose $\dot{W}= \{\dot{W}(t,x), t \in [0,+\infty),x \in \overline{D}\}$
is a 2-parameter white noise defined on the complete probability space $(\Omega,\mathcal F,\mathbb{P})$.

Let us discuss the rigorous meaning of solution to (\ref{00}), and the definition of finite time
blow-up.We regard (\ref{00}) as short-hand for the following integral equation
\begin{equation}
\begin{aligned}
u(t,x)=& \int_DK(t,x-y)(J+T+1)(1+u_0(y))dy+\int_DS(t,x-y)(J+T+1)v_0(y)dy \\
& +c_1(S\ast u)(t,x)+\int_0^t\int_DS(t-s,x-y)(c_2u+f(u))\dot{W}(ds,dy),\label{201}
\end{aligned}
\end{equation}
which may only be well-defined for a small time (see below).Here $*$ denotes convolution,i.e.
\begin{equation}
(S\ast u)(t,x)\triangleq
\int_0^t\int_DS(t-s,x-y)u(s,y)dsdy.
\nonumber
\end{equation}
$K(t,x)$ and $S(t,x)$ are
the wave kernels:
\begin{equation}
\begin{aligned}
& K(t,x)\triangleq {\textstyle\sum\limits_{n \in Z}}\frac{1}{2}(\delta(x+t+nJ)+\delta(x-t+nJ)), \\
& S(t,x)\triangleq {\textstyle\sum\limits_{n \in Z}}\frac{1}{2}I_{[-t,t]}(x+nJ), \label{22}
\end{aligned}
\end{equation}
where $\delta(\cdot)$ is the delta function.

The above formula for $K(t,x)$ should be interpreted in the
sense of Schwartz distributions.
One can define a solution to (\ref{00}) in terms of distributions and then show
that such a solution exists if and only if (\ref{201}) holds.
The integral in (\ref{201}) involving $\dot{W}(\cdot,\cdot)$ should be interpreted in the
sense of Walsh's theory of martingale measures (see \cite{JBWalsh},chapter 2).By standard arguments
(e.g. see Theorem 3.2 and exercise 3.7 in Walsh \cite{JBWalsh}) (\ref{201}) has a unique continuous solution $u(t,x)$ valid
for $t < \sigma_L$, where
\begin{equation}
\sigma_L\triangleq inf\{t > 0 : \underset{x\in \overline{D}} {sup}|u(t,x)|\geq L\}
\end{equation}
and the infimum of the empty set is taken to be $+\infty.$Letting $L \rightarrow +\infty,$
we conclude that (\ref{00}) has a unique solution for $t< \sigma ,$
where
\begin{equation} \sigma \triangleq \underset{L\rightarrow +\infty}{lim}\sigma_L. \label{0221}\end{equation}

It follows that, if $ \sigma <+\infty,$then
$\underset {t \rightarrow \sigma-}{lim}\underset {x\in \overline{D}}{sup}|u(t,x)|=+\infty.$
With these definitions in place,if $\mathbb P(\sigma<+\infty)>0$,
we say that solutions to (\ref{00}) blow up in finite time with positive probability.

\subsection{Preliminaries}
We shall use the following Lemmas:
\begin{lemma}
\label{2004}
Consider the following initial value problem for the nonlinear stochastic wave equation:
\begin{equation}
\begin{array}{l}\partial^2_tu=\partial^2_xu+g_1(u)+f_1(u)\dot{W}(t,x),t\geq0,\\
u(0,x)=u_1(x),\partial_tu(0,x)=v_1(x),\\ \label{2005}
\end{array}
\end{equation}
The function $ g_1,f_1:R \rightarrow R$ are measurable and there exists a constant $L,K>0$,
such that for $ \phi_1(u)=f_1(u),g_1(u)$,
\begin{equation}
|\phi_1(u)-\phi_1(v)|\leq L|u-v|,~~|\phi_1(u)| \leq K(1+|u|),  \label{2006}
\end{equation}
 $u_1(x),v_1(x)\in C^1(R)$.Then the equation (\ref{2005}) has a unique solution,which has a H$\ddot{o}$lder continuous version.
\end{lemma}

\begin{proof}
Existence,uniqueness and H$\ddot{o}$lder continuous of the solution to the non-linear stochastic wave equation (\ref{2005})
is covered in (\cite{JBWalsh},p.323,Exercise 3.7).The proof is omit.
\end{proof}

\begin{lemma}
\label{L20}
Let $ u(t,x):I \times D \rightarrow R $ be a solution of the following
initial-boundary value problem for the nonlinear wave equation:
\begin{equation}
\label{021}
\begin{array}{l}\partial^2_t u=\partial^2_x u+\frac{\kappa^2}{4}|u|^r-\frac{c_1^2+c_2^2}{2}u,t>0,x\in D,\\
u(t,x)=0,\mbox{for }~x\in \partial D,~t\geq 0,\\
u(0,x)=u_0(x),\partial_t u(0,x)=v_0(x),
\end{array}
\end{equation}
for which $u_0(x),v_0(x) \in C^{\infty}(D),r>1$.
Then  $u(t,x) \in C^{\infty}(I \times  D)$.

\end{lemma}

\begin{proof}
Applying Proposition 3.1 of (\cite{Taylor},P433) and
Corollary 1.6 of (\cite{Taylor},P421),we can easily prove the conclusion.
\end{proof}

\begin{lemma}
\label{2017}(\cite{RTGlassey},P185,Lemma 1.1)
Let $\phi(t)\in C^2$satisfy
$$\ddot{\phi}\geq h(\phi),t\geq 0,$$
with $\phi(t)=\alpha>0,\dot{\phi}(0)=\beta>0.$ Suppose that $h(s)\geq 0$
for all $s\geq \alpha.$ Then

(1) $\dot{\phi}(t)>0$ wherever $\phi(t)$ exists;and

(2) the inequality
$$
t\leq
\int^{\phi(t)}_\alpha[\beta^2+2\int^{s}_\alpha h(\xi)d\xi]^{\frac{-1}{2}}ds
$$
obtains.
\end{lemma}

We define the collection
$$
\{\dot{W}(A)=\int_A \dot{W}(dsdx)| A \mbox{ be Borel subset of } [0,t]\times \overline{D}\},
$$
as a centered Gaussian random field with covariance given by
$$
\mathbb E[\dot{W}(A)\dot{W}(B)]=\pi(A \cap B),
$$
where $\pi$ denotes the Lebesgue measure on $\mathbb{\mathbb{R}}_{+}\times \overline{D}.$

We define for each $t>0$ the $\sigma$-algebra
$$
\mathcal G_{t}=\sigma \{\dot{W}(A)| A \mbox{ be Borel subset of } [0,t]\times \overline{D}\} \vee \mathcal N,
~~\mathcal F_{t}= {\underset {s>t} {\cap } }\mathcal G_{s},~~ t\geq 0,
$$
where $\mathcal N$ are the totality of $\mathbb{P}$-null sets of $\mathcal F(=\mathcal B(\Omega))$.
Then,it is clear that the filtered complete probability space $(\Omega,\mathcal F,\{\mathcal F_{t}\}_{t\geq 0},\mathbb{P})$
satisfies usual hypotheses.

In order to express the idea of the proof clearly,let us define the following concept.
\begin{defn}
\label{d0201}
Let $\xi(\omega)$ be a random variable defined on the complete probability space $(\Omega,\mathcal F,\{\mathcal F_{t}\}_{t\geq 0},\mathbb{P})$,
$\Omega_{\delta}$ be an open set,
$\Omega \supseteq \Omega_{\delta},0<\mathbb{P}(\Omega_{\delta})\leq \delta.$
$\mathbb E_{\delta}\xi=\int_{\Omega-\Omega_{\delta}}\xi(\omega)\mathbb{P}(d\omega)$
is called the partial expectation of $\xi$,if $\int_{\Omega-\Omega_{\delta}}|\xi(\omega)|\mathbb{P}(d\omega)<+\infty.$
\end{defn}

The partial expectation operator-$\mathbb E_{\delta}$ has the following proposition:
\begin{prop}
\label{p0215}
Let $\dot{W}= \{\dot{W}(t,x), t \in [0,+\infty),x \in \overline{D}\}$
be a 2-parameter white noise,and $(\Omega,\mathcal F,\{\mathcal F_{t}\}_{t\geq 0},\mathbb{P})$
be a complete probability space,
$\{v(t,x;\cdot),(t,x)\in [0,+\infty)\times \overline{D}\}$ is predictable and
$\mathbb E_{\delta}\int_0^t\int_D [v(s,y;\omega)]^2dsdy<+\infty,0\leq t\leq T $.
Then,it follows that
\begin{equation}
\mathbb E_{\delta} \int_0^t\int_D v(s,y;\omega)\dot{W}(dsdy)=0,0\leq t\leq T, \label{0215}
\end{equation}
and
\begin{equation}
\mathbb E_{\delta} (\int_0^t\int_D v(s,y;\omega)\dot{W}(dsdy))^2 =
\mathbb E_{\delta} \int_0^t\int_D [v(s,y;\omega)]^2dsdy,0\leq t\leq T. \label{0216}
\end{equation}
Moreover,Burkholder's inequality and Kolmogorov Lemma on $\mathbb E_{\delta}$-version also hold.
\end{prop}

Proofs of the above results are straightforward by the definitions of stochastic integral
and the indicator of $\Omega-\Omega_{\delta}$ be an $\{\mathcal F_{t}\}_{t\geq 0}$-predictable random process.

Let us introduce the following lemma that will be used later.
\begin{lemma}
\label{L212}(\cite{Ikeda},P2,Proposition 2.1)
Let $\Omega$ be an Polish space,$\mathcal B(\Omega)$ be the topological $\sigma$-field,
and $\mathbb P $ be a probability on $(\Omega,\mathcal B(\Omega))$.
Then,for every $B\in \mathcal B(\Omega),$
\begin{equation}
 \mathbb P (B)=\underset {B\subset G,G:open} {inf}\mathbb P (G). \label{0213}
 \end{equation}
\end{lemma}

\section{Blow-up for initial data}
Let $\psi(x)$ denote the first eigenfunction for the problem
$({\frac{d}{dx}})^2\psi(x)+\mu\psi(x)=0,~x\in D,$
under the Dirichlet condition $\psi(x)=0$ on $ \partial D,$ and let $\mu=\mu_1$ be the corresponding first eigenvalue,
i.e. $ \mu_1=({\frac{\pi}{J}})^2,\psi(x)={\frac{\pi}{2J}}sin{\frac{\pi x}{J}}$.\\

We assume that\\
H1) $0\leq u_0(x) \in C^{\infty}(D),0\leq v_0(x)\in C^{\infty}(D),$there exist $x_0\in D$ such that $v_0(x_0)>0.$\\
H2) $ r >1,\lambda_1 \triangleq \mu_1+\frac{c_1^2+c_2^2}{2},\alpha\triangleq\int_D\psi(x)u_0(x)dx \geq (\frac{4\lambda_1}{\kappa^2})^{\frac{1}{r-1}}
,\beta\triangleq\int_D\psi(x)v_0(x)dx,$and
\begin{equation}
\begin{aligned}
T\triangleq &
\int^{+\infty}_\alpha[\lambda_1\alpha^2+\beta^2-\lambda_1s^2+\frac{\kappa^2}{2+2r}(s^{(r+1)}-\alpha^{(r+1)})]^{\frac{-1}{2}}ds.  \label{02}
\end{aligned}
\end{equation}

Consider the initial value problem for the nonlinear stochastic wave equation:
\begin{equation}
\begin{aligned}
&\partial^2_tu=\partial^2_xu+c_1u+(c_2u+f(u)) \dot{W}(t,x),t>0,x\in \overline{D},\\
&u(0,x)=(J+T+1)(1+u_0(x)),\partial_tu(0,x)=(J+T+1)v_0(x).\\
\end{aligned} \label{01}
\end{equation}
Here $f(u)$ is locally Lipschitz function and satisfy
$$|f(u)|\geq \kappa |u|^r,r>1,\kappa>0,$$
$c_1,c_2 $ are given constants,$T$ is given by (\ref{02}),$D=(0,J)\subset R $,
$\dot{W}(t,x)$ is 2-parameter white noise.

The main result of this article is the following.
\begin{theorem}
\label{T1}
The solution of (\ref{01}),for which $H1)$ and $H2)$ are satisfied,will blow up in bounded time with positive probability,
more precisely,for all $\varepsilon>0,$
\begin{equation}
\mathbb P(\sigma<T+\varepsilon)>0,
\nonumber
\end{equation}
\end{theorem}
where $T$  is given by (\ref{02}).

Before proving this theorem,the following lemmas are introduced.
\begin{lemma}
\label{L0}
Let $ u(t,x)$ be a solution of the following initial-boundary value problem for the nonlinear wave equation:
\begin{equation}
\label{031}
\begin{array}{l}\partial^2_t u=\partial^2_x u+\frac{\kappa^2}{4}|u|^r-\frac{c_1^2+c_2^2}{2}u,t>0,x\in D,\\
u(t,x)=0,\mbox{for }~x\in \partial D,~t\geq 0,\\
u(0,x)=u_0(x),\partial_t u(0,x)=v_0(x),
\end{array}
\end{equation}
for which $H_1)~and~H_2)$ are satisfied.
Then
\begin{equation}
\underset{t\rightarrow t_1-}{lim}\underset{x\in \overline{D}}{sup}|u(t,x)|=+\infty
\nonumber
\end{equation}
for some finite $t_1 \leq T $ ,where $T$  is given by (\ref{02}).
\end{lemma}

\begin{proof}
The solution $u(t,x)$ satisfies the following nonlinear integral equation:
\begin{equation}
u(t,x)=S\ast (\frac{\kappa^2}{4} |u|^r-\frac{c_1^2+c_2^2}{2}u)(t,x)+I(t,x),
\end{equation}
where $I(t,x)=\int_DK(t,x-y)u_0(y)dy+\int_DS(t,x-y)v_0(y)dy$ and $*$ denotes convolution,i.e.
\begin{equation}
(S\ast \eta)(t,x)\triangleq
\int_0^t\int_DS(t-s,x-y)\eta(s,y)dsdy.
\nonumber
\end{equation}

By Lemma \ref{L20},we have $u(t,x) \in C^2$.Let
$ \phi(t)=\int^J_0\psi(x)u(t,x)dx,$
multiply (\ref{031}) by $\psi$ and integrate over $D,$ we obtain
$$
\int^J_0\psi u_{tt}dx=\ddot{\phi}=\int^J_0\psi u_{xx}dx+\frac{\kappa^2}{4}\int^J_0 \psi |u|^rdx-\frac{c_1^2+c_2^2}{2}\int^J_0 \psi udx.
$$

By Jensen's inequality,we have $\int^J_0 \psi |u|^rdx \geq |\phi|^r$,since $\int^J_0\psi(x)dx=1,\psi(x)\geq 0,x \in D.$
Using integration by parts and the boundary conditions satisfied by $u$ and $\psi$,we see that
$$
\int^J_0\psi u_{xx}dx=\int^J_0 u \psi_{xx}dx=-\mu_1\int^J_0 u \psi dx=-\mu_1\phi.
$$
Thus we arrive at
$$
\ddot{\phi}+(\mu_1+\frac{c_1^2+c_2^2}{2}) \phi\geq \frac{\kappa^2}{4}|\phi|^r
$$
with
$$
\phi(0)=\int^J_0\psi(x)u(0,x)dx=\alpha>0, \dot{\phi}(0)=\int^J_0\psi(x)u_{t}(0,x)dx=\beta>0.
$$
Hypothesis H2) implies that Lemma \ref{2017} is applicable with $h(s)=\frac{\kappa^2}{4}s^r-\lambda_1s$;
therefore
$$
t\leq
\int^{\phi(t)}_{\alpha} [\lambda_1 \alpha^2+\beta^2-\lambda_1 s^2+\frac{\kappa^2}{2+2r}(s^{(r+1)}-\alpha^{(r+1)})]^{\frac{-1}{2}}ds ,
$$
and thus $\phi(t)$ develops a singularity in a finite time $t\leq T$,where
$$
T=\int^{\infty}_{\alpha} [\lambda_1\alpha^2+\beta^2-\lambda_1s^2+\frac{\kappa^2}{2+2r}(s^{(r+1)}-\alpha^{(r+1)})]^{\frac{-1}{2}}ds.
$$

Finally,since $\phi(t)>0,$ we have
$$
\phi(t)=|\phi(t)|=|\int^J_0\psi(x)u(t,x)dx|\leq \underset{x\in \overline{D}}{sup}|u(t,x)|,
$$
which proves the Lemma.
\end{proof}

Now let us prove the following comparison Lemma which could be also of interest in itself.
\begin{lemma}
\label{302}
Let $U(t,x)$ satisfy equation (\ref{031}),$0<t_f<t_1$,define

\begin{equation}
\begin{array}{l}
\mathcal V=\Bigl\{v(\cdot,\cdot)\in C([0,t_f]\times \overline{D})\Bigl|
v(t,x)>S\ast (\frac{\kappa^2}{4} |v|^r-\frac{c_1^2+c_2^2}{2}v)(t,x)+I(t,x)\Bigr \},
\end{array} \nonumber
\end{equation}
where $I(t,x)=\int_DK(t,x-y)u_0(y)dy+\int_DS(t,x-y)v_0(y)dy$
and $t_1$ is given by the Lemma \ref{L0}.
Then the set $\mathcal V$ has the following properties:
\begin{equation}
\begin{array}{l}
(1)~~\mathcal V \mbox {~is a nonempty convex set},\\\\
(2)~~v(t,x)>U(t,x),\mbox {for any } v(t,x)\in \mathcal V,(t,x) \in [0,t_f]\times \overline{D}.
\end{array} \nonumber
\end{equation}
\end{lemma}

\begin{proof} (1)First of all,noting $U(t,x)=S\ast (\frac{\kappa^2}{4} |U|^r-\frac{c_1^2+c_2^2}{2}U)(t,x)+I(t,x),$
if let $M=\underset {(t,x)\in[0,t_f]\times \overline{D}}{max }U(t,x),$
$f_0(t)=exp\{[\frac{r\kappa^2}{4}(M+1)^{r-1}-\frac{c_1^2+c_2^2}{2}]J(t-t_f)\},0\leq t \leq t_f,$
define $v_1(t,x)=U(t,x)+ f_0(t)$,then we have
\begin{equation}
\begin{array}{l}
I(t,x)+S\ast (\frac{\kappa^2}{4} |v_1|^r-\frac{c_1^2+c_2^2}{2}v_1)(t,x)\\
~~\leq I(t,x)+\frac{\kappa^2}{4}\int_0^t\int_DS(t-s,x-y)(|U(s,y)|^r+r(M+1)^{r-1}f_0(s))dsdy \\
~~~~~~~~~~~~~~-\frac{c_1^2+c_2^2}{2}\int_0^t\int_DS(t-s,x-y)(U(s,y)+f_0(s))dsdy \\
~~\leq U(t,x)+[\frac{r\kappa^2}{4}(M+1)^{r-1}-\frac{c_1^2+c_2^2}{2}]J\int_0^tf_0(s)ds \\
~~< U(t,x)+f_0(t)=v_1(t,x),
\end{array} \nonumber
\end{equation}
for $(t,x) \in [0,t_f]\times \overline{D}.$
Thus $v_1(t,x) \in \mathcal V, $ i.e. $\mathcal V $ is a nonempty set.
Next,by Jensen's inequality,It is easy to see that $\mathcal V $ is a convex set.
In order to prove (2) we use (1),if (2) is false,
then there exists $v_2(t,x)\in \mathcal V $ and $(t_e,x_e)\in [0,t_f]\times \overline{D}$,
such that $v_2(t_e,x_e)<U(t_e,x_e)$,
since $U(t,x)$ and $\mathcal V $  is non-intersect.
Select $v_1(t,x)\in \mathcal V $
as above,then there exists $0<\theta<1,$
such that $U(t_e,x_e)=\theta v_1(t_e,x_e)+(1-\theta)v_2(t_e,x_e)$,
thus we obtain $U(t,x)$ intersects with $\mathcal V $ ,
this is a contradiction,the proof is complete.
\end{proof}

We present the following result that
the essential supremum of the solution of (\ref{01})
over a subset $\Omega-\Omega_\delta$ of the probability space,will blow up in finite time.
\begin{lemma}
\label{L033}
Let $ u(t,x)$ is the solution of (\ref{01}),for which $H1)$ and $H2)$ are satisfied.
Then,for given an open set $\Omega_{\delta} \subset \Omega,0<\mathbb{P}(\Omega_{\delta})\leq \delta
\leq \frac{1}{3},$ it follows that
\begin{equation}
\underset{t\rightarrow t_0-}{lim}
\underset {x\in \overline{D}} {max} \underset {\Omega-\Omega_{\delta}} {esssup}|u(t,x)|=+\infty,
\nonumber
\end{equation}
for some bounded time $t_0=t_0(\Omega_{\delta}) \leq T $,where $T$ is given by (\ref{02}).
\end{lemma}

Before proving Lemma \ref{L033},let us remark on some details of our approach.

If $\underset {[0,T]\times \overline{D}} {max} \underset {\Omega-\Omega_{\delta}} {esssup}|u(t,x)|<+\infty $,
by the Definition \ref{d0201} of the partial expectation operator-$\mathbb E_{\delta}$,it is clear that
\begin{equation}
\mathbb E_{\delta} \underset {x \in \overline{D}} {sup} u(t,x)^{2}\leq \underset {[0,T]\times \overline{D}} {max} \underset {\Omega-\Omega_{\delta}} {esssup}|u|^{2}<+\infty, t\in [0,T] \label{0300}
\end{equation}
and
\begin{equation}
\begin{aligned}
& \mathbb E_{\delta} \int_0^t \int_D S(t-s,x-y) |f(u(s,y))|^{2}dsdy <+\infty,\\
& \mathbb E_{\delta} \int_0^t \int_D S(t-s,x-y) |u(s,y)|^{q}dsdy <+\infty,\\
\label{300}
\end{aligned}
\end{equation}
for $(t,x)\in [0,T]\times \overline{D},1 \le q <+\infty$.
It follows that the equation (\ref{01}) exists a continuous local solution
$u(t,x)$ on $[0,T]\times \overline{D},$ for $\omega \in (\Omega-\Omega_\delta)$,
since (by Proposition \ref{p0215}) the Burkholder's inequality and Kolmogorov Lemma on $\mathbb E_{\delta}$-version hold.
Moreover,by(\ref{0300}),using dominated convergence theorem,
we can carry out
\begin{equation}
\mathbb E_{\delta} u(t,x)^{2} \in C([0,T]\times \overline{D}). \label{0330}
\end{equation}

We now turn to the proof of Lemma \ref{L033}.
\begin{proof}
Suppose that $\underset {[0,T]\times \overline{D}} {max} \underset {\Omega-\Omega_{\delta}} {esssup}|u(t,x)|<+\infty $.
If define $I(t,x)=\int_D K(t,x-y)u_0(y)dy+\int_D S(t,x-y)v_0(y)dy,$
by (\ref{0215}),(\ref{0216}),(\ref{201}) and (\ref{300}),using Jensen's inequality and Schwarz's inequality,
noting $(c+d)^2 \geq \frac{1}{2}c^2-d^2$,$(c+d)^2 \geq \frac{1}{2(J+t+1)^2}c^2-\frac{1}{(J+t+1)^2}d^2$,
$S(t,x)^2\geq \frac{1}{2}S(t,x)$ and $\int_0^t\int_D S(t-s,x-y)dsdy \leq \frac{(J+t)^2}{2},$ then we have
\begin{equation}
\begin{aligned}
\mathbb E_{\delta} u(t,x)^{2}& =\mathbb E_{\delta} [(J+T+1)I(t,x)+(J+T+1)+c_1S*u]^2 \\
&~~~~~~~~~~~~~+\mathbb E_{\delta} \int_0^t\int_D [S(t-s,x-y) (c_2u(s,y)+f(u)]^{2} dsdy \\
&\geq \mathbb E_{\delta} [(J+T+1)I(t,x)+(J+T+1)]^2/[2(J+t+1)^2] \\
&~~~~~~~~~~~~~-\mathbb E_{\delta} (c_1S*u)^2/[(J+t+1)^2]\\
&~~~~~~~~~~~~~+\mathbb E_{\delta} \int_0^t\int_D [S(t-s,x-y) (c_2u(s,y)+f(u)]^{2} dsdy \\
&>(1-\delta)[I(t,x)+1]^2/2-\frac{c_1^2}{2}\mathbb E_{\delta} \int_0^t\int_D S(t-s,x-y)u(s,y)^2dsdy \\
&~~~~~~~~~~~~~+\frac{1}{2}\mathbb E_{\delta} \int_0^t\int_D S(t-s,x-y) (\frac{\kappa^2}{2}|u(s,y)|^{2r}-(c_2u(s,y))^2) dsdy \\
&> I(t,x)-\frac{c_1^2}{2}\mathbb E_{\delta} \int_0^t\int_D S(t-s,x-y) u(s,y)^2 dsdy \\
&~~~~~~~~~~~~~+\frac{1}{2}\mathbb E_{\delta} \int_0^t\int_D S(t-s,x-y) (\frac{\kappa^2}{2}|u(s,y)|^{2r}-(c_2u(s,y))^2) dsdy \\
&\geq I(t,x)+\frac{\kappa^2}{4}\int_0^t\int_D S(t-s,x-y) (\mathbb E_{\delta} u(s,y)^{2})^r dsdy \\
&~~~~~~~~~~~~~-\int_0^t\int_D S(t-s,x-y) \frac{c_1^2+c_2^2}{2}\mathbb E_{\delta} u(s,y)^{2} dsdy \\
&=I(t,x)+S\ast (\frac{\kappa^2}{4} |\mathbb E_{\delta} u^{2}|^r-\frac{c_1^2+c_2^2}{2}\mathbb E_{\delta} u^{2})(t,x)
,(t,x)\in [0,T]\times \overline{D}.\label{0301}
\end{aligned}
\end{equation}
Now,combining (\ref{0330}) and (\ref{0301}),using Lemma \ref{302},we obtain
\begin{equation}
\mathbb E_{\delta} u(t,x)^{2}>U(t,x),(t,x) \in [0,t_f]\times \overline{D}. \nonumber
\end{equation}
Thus we arrive at
\begin{equation}
\mathbb E_{\delta} \underset {x \in \overline{D}} {sup} u(t_f,x)^{2}\geq
\underset {x \in \overline{D}} {sup} \mathbb E_{\delta} u(t_f,x)^{2}\geq
\underset {x \in \overline{D}}{sup} U(t_f,x).
\end{equation}
Let $t_f \rightarrow t_1-,$ by Lemma \ref{L0},we have
$\underset{t_f\rightarrow t_1-}{lim}\mathbb E_{\delta} \underset {x \in \overline{D}} {sup} u(t_f,x)^{2}=+\infty,$
this contradicts with (\ref{0300}).Thus there exists some bounded time $t_0=t_0(\Omega_{\delta}) \leq t_1 \leq T,$
such that
$\underset{t\rightarrow t_0-}{lim}
\underset {x\in \overline{D}} {max} \underset {\Omega-\Omega_{\delta}} {esssup}|u(t,x)|=+\infty.$
The proof is complete.
\end{proof}

The following lemma tells us a fact
that the essential infimum of the blow-up time of the solution of (\ref{00}) is bounded:
\begin{lemma}
\label{L0312}
Let $ \sigma $ is given by (\ref{0221}) and $\tau=\underset {\Omega} {essinf} \sigma,$ then we have $\tau \leq T.$
\end{lemma}
\begin{proof}
In order to express the corresponding notations clearly,we replace $u(t,x),\sigma$ by $u(t,x;\cdot),\sigma(\cdot).$
Recall that $ \sigma_n(\omega)= inf\{t > 0 : \underset{x\in \overline{D}} {sup}|u(t,x;\omega)|\geq n\},$ by Lemma \ref{L033},
for given an open set $\Omega_{\delta} \subset \Omega,0<\mathbb{P}(\Omega_{\delta})\leq \delta
\leq \frac{1}{3},$if define
\begin{equation}
E_n= \{\omega\in \Omega-\Omega_{\delta} : \sigma_n(\omega)<t_0\},\mbox{ for all }  n\in \mathbb{N},\label{037}
\end{equation}
then we have
$E_{n+1}\subset E_n\in \mathcal B(\Omega)|_{(\Omega-\Omega_{\delta})}$ and $\mathbb{P}(E_n)>0,\mbox{ for all } n\in \mathbb{N}$.

It is easy to show that $\Omega-\Omega_{\delta}$ is also a Polish space with the metric $\gamma$
since $\Omega_{\delta}$ be a open set.The closure of $A$ in $\Omega-\Omega_{\delta} $,denoted by $(\bar{A})_{\delta}$,
and the boundary set of $A$ in $\Omega-\Omega_{\delta} $,denoted by $\partial_{\delta}{(A)}$,
if define $K_{n}=(\bar{E}_{n+1})_{\delta},n\in \mathbb{N}$,then we have
\begin{equation}
K_n\subset E_n,\mathbb{P}(K_n)\geq \mathbb{P}(E_{n+1})>0 \mbox{ and}\overset{n} {\underset {m=1} {\cap } }K_m\subset \overset{n} {\underset {m=1} {\cap } }E_m,
\mbox{ for all } n \in \mathbb{N}. \label{038}
\end{equation}
Indeed,by the continuity of $u(t,x;\cdot)$,
the boundary set of $E_{n+1}$ in $\Omega-\Omega_{\delta} $ satisfy
$$\partial_{\delta}{(E_{n+1})}\subseteq \{\omega \in \Omega-\Omega_{\delta}|\underset{t\in [0,t_0]} {sup}\underset{x\in \overline{D}} {sup}|u(t,x;\omega)|=n+1\}\subset E_n.$$
This leads to $ K_n\subset E_n,$ for all $ n \in \mathbb{N}.$

Moreover,noting that $\Omega-\Omega_{\delta}$ is a Polish space with the metric $\gamma$,it follows that there exist
$\omega_f \in \overset{+\infty} {\underset {n=1} {\cap } }K_n \neq \emptyset$,
since $\{K_n\}_{n\in \mathbb{N}}$ is a closed set sequence
and $K_{n+1}\subset K_{n},$for all $ n \in \mathbb{N}$.
Due to (\ref{038}) and $K_{n} \subset E_1\subseteq \Omega-\Omega_{\delta},\mbox{ for all } n \in \mathbb{N}$,
it is evident to see that $\omega_f \in \Omega-\Omega_{\delta}$ and
$\omega_f \in \overset{+\infty} {\underset {n=1} {\cap } }E_n$.
If plug $\omega_f $ back into (\ref{037}),then we obtain
\begin{equation}
\sigma_n(\omega_f) <t_0, \mbox{ for all } n\in \mathbb{N}.
\end{equation}
Let $n\rightarrow +\infty$,it is now obvious that
\begin{equation}
\underset {\Omega-\Omega_{\delta}} {inf}\sigma \leq \sigma(\omega_f)
=\underset {n\rightarrow +\infty} {lim} \sigma_n(\omega_f) \leq t_0=t_0(\Omega_{\delta})\leq T. \label{0318}
 \end{equation}

In addition,it is evident that,according to the conclusion of Lemma \ref{L212} and
the definition of $\mbox{essinf},$
\begin{equation}
\tau=\underset {E_0\subset \Omega, \mathbb P(E_0)=0} {sup}(\underset {\Omega-E_0} {inf}\sigma)
\leq \underset {\Omega_{\delta}:open, 0<\mathbb{P}(\Omega_{\delta})\leq \delta
\leq \frac{1}{3}} {sup}
(\underset {\Omega-\Omega_{\delta}} {inf}\sigma).  \label{0320}
 \end{equation}

Combining (\ref{0318}) and (\ref{0320}),we get
\begin{equation}
\tau \leq \underset {\Omega_{\delta}:open, 0<\mathbb{P}(\Omega_{\delta})\leq \delta
\leq \frac{1}{3}} {sup}
(\underset {\Omega-\Omega_{\delta}} {inf}\sigma) \leq T.
 \end{equation}

This completes the proof of Lemma \ref{L0312}.
\end{proof}

We are now in a position to prove Theorem \ref{T1}.
\begin{proof}
Suppose that on the contrary,there exists $\varepsilon_0>0,$ such that
$\mathbb P(\sigma<T+\varepsilon_0)=0, $
then according to the definition of $\mbox{essinf},$ we have
\begin{equation}
\tau=\underset {E_0\subset \Omega, \mathbb P(E_0)=0} {sup}(\underset {\Omega-E_0} {inf}\sigma)
\geq T+\varepsilon_0>T. \end{equation}
However,by Lemma \ref{L0312},we have $\tau \leq T,$
this leads to a contradiction.
Thus,for all $\varepsilon>0$,we have $\mathbb P(\sigma<T+\varepsilon)>0,$ the proof is complete.
\end{proof}

%% \section*{Acknowledgments}
%% The author thanks the anonymous authors whose work largely
%% constitutes this sample file. He also thanks the INFO-TeX mailing
%% list for the valuable indirect assistance he received.

\end{document}